\documentclass[10pt,reqno]{amsart}
   \usepackage[centertags]{amsmath}
   \usepackage{amsfonts}
   \usepackage{amsmath}
   \usepackage{amssymb}
   \usepackage{amsthm}
   \usepackage{newlfont}
   \usepackage[latin1]{inputenc}%%%%para que reconozca las widetildes.

\usepackage{caption}
\usepackage{subcaption}

\usepackage{multirow, bigdelim}

\usepackage{epsfig}
\usepackage{pst-all}
\usepackage{pstricks}
\usepackage{pgf}
\usepackage{mathrsfs}
\usepackage{hyperref}
\usepackage{bm}

\usepackage{graphicx}

\usepackage{bm}

\usepackage{mathtools}
\usepackage{ifmtarg}

\makeatletter

    \newcommand\contFrac{\@ifstar{\@contFracStar}{\@contFracNoStar}}

    \def\singleContFrac#1#2{%
        \begin{array}{@{}c@{}}%
            \multicolumn{1}{c|}{#1}%
            \\%
            \hline%
            \multicolumn{1}{|c}{#2}%
        \end{array}%
    }

    \def\@contFracNoStar#1{%
        \mathchoice{% * Display style
            \@contFracNoStarDisplay@#1//\@nil%
        }{%           * Text style
            \@contFracNoStarInline@#1//\@nil%
        }{%           * Script style
            \@contFracNoStarInline@#1//\@nil%
        }{%           * Script script style
            \@contFracNoStarInline@#1//\@nil%
        }%
    }

    \def\@contFracNoStarDisplay@#1//#2\@nil{%
        \@ifmtarg{#2}{%
            #1%
        }{%
            #1+\cfrac{1}{\@contFracNoStarDisplay@#2\@nil}%
        }%
    }

        \def\@contFracNoStarInline@#1//#2\@nil{%
            \@ifmtarg{#2}{%
                #1%
            }{%
                #1 \@@contFracNoStarInline@@#2\@nil%
            }%
        }
        \def\@@contFracNoStarInline@@#1//#2\@nil{%
            \@ifmtarg{#2}{%
                + \singleContFrac{1}{#1}%
            }{%
                + \singleContFrac{1}{#1} \@@contFracNoStarInline@@#2\@nil%
            }%
        }

    \def\@contFracStar#1{%
        \mathchoice{% * Display style
            \@contFracStarDisplay@#1////\@nil%
        }{%           * Text style
            \@contFracStarInline@#1//\@nil%
        }{%           * Script style
            \@contFracStarInline@#1//\@nil%
        }{%           * Script script style
            \@contFracStarInline@#1//\@nil%
        }%
    }

    \def\@contFracStarDisplay@#1//#2//#3\@nil{%
        \@ifmtarg{#2}{%
            #1%
        }{%
            #1 + \cfrac{#2}{\@contFracStarDisplay@#3\@nil}%
        }%
    }

        \def\@contFracStarInline@#1//#2\@nil{%
            \@ifmtarg{#2}{%
                #1%
            }{%
                #1 \@@contFracStarInline@@#2\@nil%
            }%
        }
        \def\@@contFracStarInline@@#1//#2//#3\@nil{%
            \@ifmtarg{#3}{%
                - \singleContFrac{#1}{#2}%
            }{%
                - \singleContFrac{#1}{#2} \@@contFracStarInline@@#3\@nil%
            }%
        }
\makeatother

\def\cFrac#1#2{%
\begin{array}{@{}c@{}}\multicolumn{1}{c|}{#1}\\%
\hline\multicolumn{1}{|c}{#2}\end{array}}

\allowdisplaybreaks[3]
\theoremstyle{plain}
\newtheorem{theorem}{Theorem}[section]

\newtheorem{proposition}[theorem]{Proposition}

\theoremstyle{definition}

\theoremstyle{remark}
\newtheorem{remark}[theorem]{Remark}
\newtheorem*{remark*}{Remark}

\numberwithin{equation}{section}

\newcommand\D{\displaystyle}

\newcommand\NN{{\mathbb N}}

\title[BDC on a spider: spectral analysis and RA factorization]{Birth-death chains on a spider: spectral analysis\\ and reflecting-absorbing factorization}

\author{Manuel D. de la Iglesia}
\address{Manuel D. de la Iglesia\\
Instituto de Matem\'aticas, Universidad Nacional Aut\'onoma de M\'exico, Circuito Exterior, C.U., 04510, Ciudad de M\'exico, M\'exico.}
\email{mdi29@im.unam.mx}

\author{Claudia Juarez}
\address{Claudia Juarez\\
Instituto de Investigaciones en Matem\'aticas Aplicadas y en Sistemas, Universidad Nacional Aut\'onoma de M\'exico, Circuito Escolar 3000, C.U., 04510, Ciudad de M\'exico, M\'exico.}
\email{claudiajrz@sigma.iimas.unam.mx}
\date{\today}
\thanks{
}

\thanks{This work was partially supported by PAPIIT-DGAPA-UNAM grant IN104219 (M\'exico) and CONACYT grant A1-S-16202 (M\'exico).}

\date{\today}



\begin{document}

\maketitle

\begin{abstract}
We consider discrete-time birth-death chains on a spider, i.e. a graph consisting of $N$ discrete half lines on the plane that are joined at the origin. This process can be identified with a discrete-time quasi-birth-death process on the state space $\NN_0 \times \{1, 2, \dots, N\},$ represented by a block tridiagonal transition probability matrix. We prove that we can analyze this process by using spectral methods and obtain the $n$-step transition probabilities in terms of a weight matrix and the corresponding matrix-valued orthogonal polynomials (the so-called Karlin-McGregor formula). We also study under what conditions we can get a reflecting-absorbing factorization of the birth-death chain on a spider which can be seen as a stochastic UL block factorization of the transition probability matrix of the quasi-birth-death process. With this factorization we can perform a discrete Darboux transformation and get new families of ``almost'' birth-death chains on a spider. The spectral matrix associated with the Darboux transformation will be a Geronimus transformation of the original spectral matrix. Finally, we apply our results to the random walk on a spider, i.e. with constant transition probabilities.
\end{abstract}

\section{Introduction}

In \cite{IMcK} (Section 4.2, Problem 1), It\^{o} and McKean proposed an elementary but interesting diffusion process which they called \emph{skew Brownian motion}. After that, in 1978, Walsh \cite{Walsh} characterized this process as a Brownian motion with excursions around zero in random directions on the plane which takes values in $[0,2\pi)$ and call it a diffusion with a discontinuous local time. This diffusion is now called \emph{Walsh's Brownian motion}. Later, in 1989, Barlow, Pitman and Yor \cite{BPY} considered this process as a motion which lives on $N$ half lines on the plane, called \emph{legs} from now on, that are joined at the origin, called the \emph{body} of the spider. More recently, in 2003, Evans and Sowers \cite{ES} considered the same construction but with different methods and used the name \textit{Walsh's spider} for this process. In a few words, the process behaves like a regular Brownian motion on each one of the legs and once it reaches the body it continues on any of the $N$ legs with a given probability. For some other results related with this process the reader can consult \cite{Harrison, Lej, PPL, VY}.

\smallskip

If we replace the Brownian motions with simple symmetric random walks on the legs we get a discrete version of Walsh's spider, also called a \emph{random walk on a spider}. Hajri \cite{H} studied this discrete version as an approximation of the Walsh's spider. Some other results related with stochastic differential equations, limit theorems or weak convergence distributions can be found in \cite{Cherny, Csetal1, Csetal2, LeGall, Seol}. In Example 3.5 of \cite{DRSZ}, Dette et al treated the discrete version of Walsh's spider as a quasi-birth-death process with state space $\NN_0 \times \{1, 2, \dots, N\}$, which is a generalization of the birth-death process but allowing transitions between all the states of the second component (or phases). They studied the case of the symmetric random walk on a spider (which they call a tree) and gave an explicit expression of the corresponding spectral matrix, which was improved later by Gr\"{u}nbaum in \cite{G1} (see also \cite{G5}). For more information about quasi-birth-death processes the reader can consult \cite{LaR, Neu}.

\smallskip

In this paper we consider general discrete-time birth-death chains on a spider, allowing transition probabilities to depend on the state and on the leg in which the particle is living. This process behaves like a regular birth-death chain in each of the legs, but once it reaches the body of the spider it continues towards any of the $N$ legs with a given probability (including remaining at the body). In Section \ref{sec2} we make a precise definition of this process. In order to perform a spectral analysis it will be more convenient to treat this process as a quasi-birth-death process with state space $\NN_0 \times \{1, 2, \dots, N\}$, where $N$ is the number of legs. The transition probability matrix will be represented by a block tridiagonal matrix $\bm{P}$ (see \eqref{tpmm} below). Using the results in \cite{DRSZ} we prove in Proposition \ref{prop1} that there always exists a weight matrix (or spectral matrix) $\bm{W}$ associated with $\bm P$ and therefore we can obtain an integral representation of the $n$-step transition probability matrix $\bm P^n$ in terms of this spectral matrix and the associated matrix-valued orthogonal polynomials, also known as the \emph{Karlin-McGregor formula} (see \cite{KMc6, DRSZ, G2} or more recently the monograph \cite{MDIB}). Using the $2\times2$ block structure of the spectral matrix $\bm W$ we are able to get in Proposition \ref{prop2} an explicit expression of the \emph{Stieltjes transform} of $\bm W$ in terms of $N$ scalar spectral measures (one for each leg), which will be very useful for the study of the example in Section \ref{sec4}.

\smallskip

Following our previous works in \cite{dIJ1,dIJ2}, we consider in Section \ref{sec3} a stochastic factorization of the birth-death chain on a spider into a reflecting and an absorbing birth-death chain to the state 0. That will mean that we will be looking for possible stochastic UL factorizations of the transition probability matrix $\bm{P}$ of the form $\bm{P}=\bm{P_R}\bm{P_A}$. The difference with \cite{dIJ2} is that now we will have $N$ free parameters, one for each leg, and we will show that each of these parameters must be bounded from below by certain continued fraction built from the transition probabilities of each leg (see Theorem \ref{thm1} below). This new method of considering stochastic factorizations was introduced for the first time in \cite{GdI3} and later exploited in \cite{GdI4, dIJ1, dIJ2}. The motivation for these stochastic factorizations is to divide the probabilistic model of the original process into two different and simpler experiments. Once we have the conditions under we can perform a reflecting-absorbing factorization, we consider a \emph{discrete Darboux transformation}, consisting in reversing the order of multiplication of the factors. Therefore we will get a family of Markov chains on a spider depending on $N$ free parameters with transition probability matrix given by $\widetilde{\bm{P}}= \bm P_A  \bm P_R$. This new stochastic matrix $\widetilde{\bm{P}}$ describes an ``almost'' birth-death chain on a spider since now there will be extra transition probabilities between the first states of each leg. However, the matrix $\widetilde{\bm{P}}$ preserves the block tridiagonal structure and it will possible to derive the corresponding spectral matrix $\widetilde{\bm{W}}$ in terms of a \textit{Geronimous transformation} of $\bm{W}$ (see \eqref{spdtr} below).

\smallskip

In Section \ref{sec4}, we apply our results to the random walk on a spider, i.e. with constant transition probabilities (including self-transitions on each state). We derive an explicit expression of the spectral matrix $\bm{W}$ computing first its Stieltjes transform from Proposition \ref{prop2} and then using the Perron-Stieltjes inversion formula. This representation (as a $2\times2$ block matrix) is different and simpler from the ones given in \cite{DRSZ,G1}. We also apply Theorem \ref{thm1} to study conditions under we get a stochastic reflecting-absorbing factorization and perform the discrete Darboux transformation. Finally, we get an explicit expression of the spectral matrix associated with the Darboux transformation $\widetilde{\bm{W}}$.

\section{Spectral analysis of birth-death chains on a spider}
\label{sec2}

For $N\in\mathbb{N}$ consider the spider graph given by
$$
\mathbb{S}_N:=\{v_N(k,m), \;k\in\mathbb{N}_0,\; m=1,\ldots,N\},
$$
where
$$
v_N(k,m)=k\;\mbox{exp}\left(\frac{2\pi i(m-1)}{N}\right),\quad i=\sqrt{-1}.
$$
The number $N$ is the number of legs of the spider $\mathbb{S}_N$. If $N=1$ then we go back to regular birth-death chains on $\mathbb{N}_0$, while if $N=2$ we have a birth-death chain on $\mathbb{Z}$ (see \cite{KMc6, dIJ1, dIJ2}). The point $v_N(0):=v_N(0,m), m=1,\ldots,N,$ will be called the body of the spider. Consider a homogeneous discrete-time birth-death chain $\{S_n, n=0,1,\ldots\}$ on a spider $\mathbb{S}_N$. The transition probabilities are given by
$$
\mathbb{P}\left[S_{n+1}=v_N(0)\,|\, S_{n}=v_N(0)\right]=\alpha_0,\quad \mathbb{P}\left[S_{n+1}=v_N(1,m)\,|\, S_{n}=v_N(0)\right]=\alpha_m,\quad m=1,\ldots,N,
$$
where $\sum_{m=0}^N\alpha_m=1,$ and
\begin{align*}
\mathbb{P}\left[S_{n+1}=v_N(k+1,m)\,|\, S_{n}=v_N(k,m)\right]&=a_{k,m},\\
\mathbb{P}\left[S_{n+1}=v_N(k,m)\,|\, S_{n}=v_N(k,m)\right]&=b_{k,m},\\ 
\mathbb{P}\left[S_{n+1}=v_N(k-1,m)\,|\, S_{n}=v_N(k,m)\right]&=c_{k,m},
\end{align*}
where $a_{k,m}+b_{k,m}+c_{k,m}=1$ for all $k\geq1$ and $m=1,\ldots,N$. A diagram of the probability transitions between the states of this process is given in Figure \ref{fig1}.
\begin{figure}[h]
\includegraphics[scale=0.25]{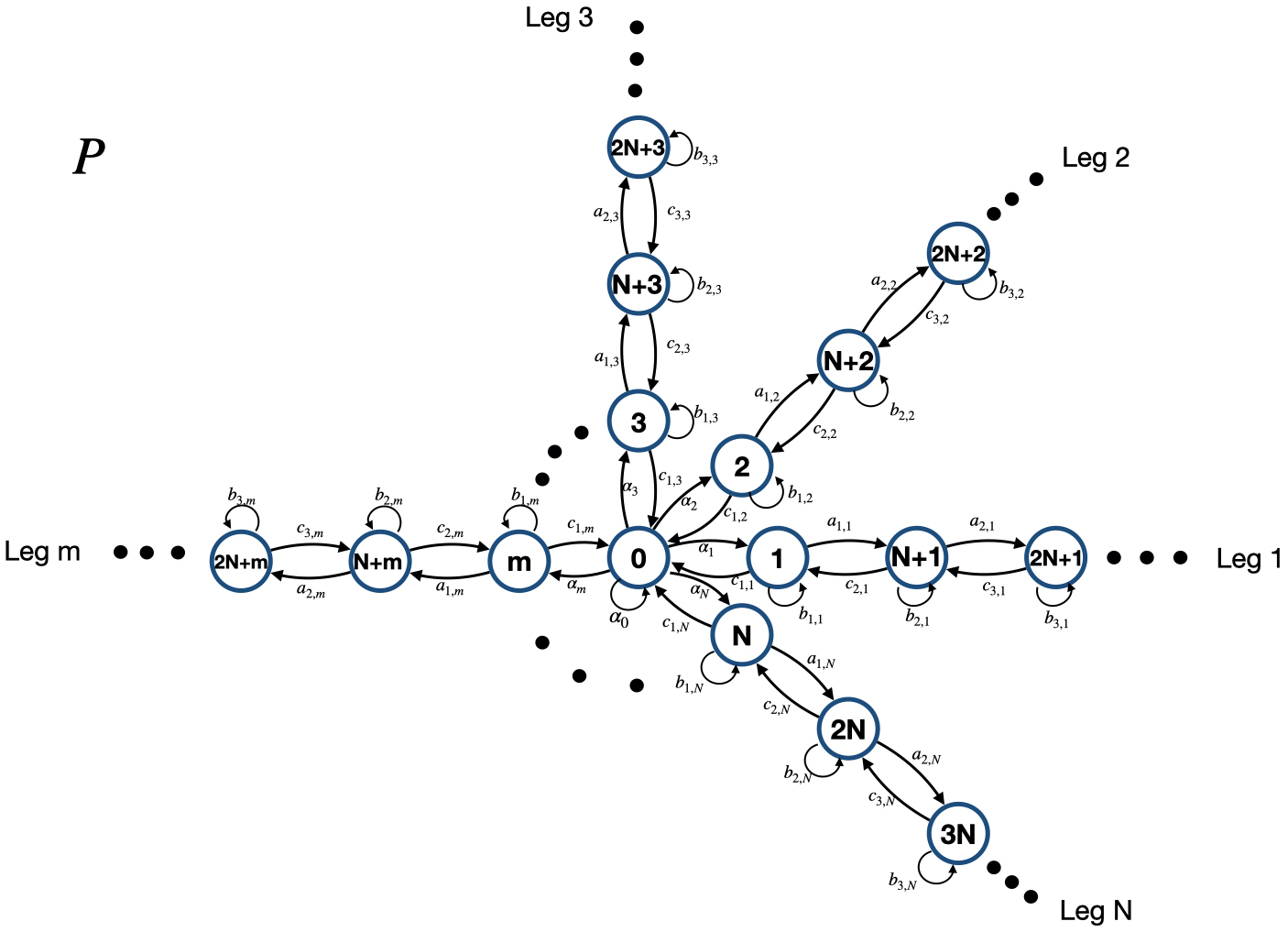}
\caption{Diagram of a birth-death chain on a spider with $N$ legs.}
\label{fig1}
\end{figure}
As it was done in \cite{DRSZ}, the birth-death chain $\{S_n, n=0,1,\ldots\}$ on a spider $\mathbb{S}_N$ can be seen as a quasi-birth-death process on the state space $\mathbb{N}_0\times\{1,2,\ldots,N\}$ (see \cite{LaR, Neu} for more information about quasi-birth-death processes). The labeling follows putting $v_N(0)$ as the origin $0$. Then the first $N$ nodes on the first circle as $1,\ldots,N,$ in a counter-clock wise fashion. The second circle with $N+1,\ldots,2N,$ and so on (see Figure \ref{fig1}). The transition probability matrix of the birth-death chain $\{S_n, n=0,1,\ldots\}$, seen as a quasi-birth-death process, is
\begin{equation}\label{tpmm}
\bm P=\begin{pmatrix}
\bm B_0&\bm A_0&&&\\
\bm C_1&\bm B_1&\bm A_1&&\\
&\bm C_2&\bm B_2&\bm A_2&\\
&&\ddots&\ddots&\ddots
\end{pmatrix},
\end{equation}
where
\begin{equation}\label{B0A0}
\bm B_0=\begin{pmatrix}
\alpha_0&\alpha_1&\alpha_2&\cdots&\alpha_{N-1}\\
c_{1,1}&b_{1,1}&0&\cdots&0\\
c_{1,2}&0&b_{1,2}&\cdots&0\\
\vdots&\vdots&\vdots&\ddots&\vdots\\
c_{1,N-1}&0&0&\cdots&b_{1,N-1}
\end{pmatrix},\quad \bm A_0=\mbox{diag}\left(\alpha_N,a_{1,1},\ldots,a_{1,N-1}\right),
\end{equation}
and $\bm A_n,\bm B_n,\bm C_n$ are the diagonal matrices
\begin{equation}\label{ABn}
\bm A_n=\mbox{diag}\left(a_{n,N},a_{n+1,1},\ldots,a_{n+1,N-1}\right), \quad\bm B_n=\mbox{diag}\left(b_{n,N},b_{n+1,1},\ldots,b_{n+1,N-1}\right),\quad n\geq1,
\end{equation}
\begin{equation}\label{Cn}
\bm C_n=\mbox{diag}\left(c_{n,N},c_{n+1,1},\ldots,c_{n+1,N-1}\right),\quad n\geq1.
\end{equation}
In the same fashion as in \cite{DRSZ, G2} we consider the matrix-valued polynomials $(\bm Q_n)_{n\geq0}$ generated by the three-term recurrence relation
\begin{align*}
x\bm Q_n(x)&=\bm A_n\bm Q_{n+1}(x)+\bm B_n\bm Q_n(x)+\bm C_n\bm Q_{n-1}(x),\quad n\geq0,\\
\bm Q_0(x)&=\bm I_N,\quad \bm Q_{-1}(x)=\bm 0,
\end{align*}
where $\bm I_N$ and $\bm0$ denote the identity and the null matrix of dimension $N\times N$, respectively (from now on whenever we write $\bm0$ we will mean the null vector or matrix which dimension will be determined by the context). These matrix-valued polynomials can be written as
\begin{equation}\label{polmat}
\bm Q_n(x)=\begin{pmatrix}
Q_{n,N}(x)&\alpha_1Q_{n,N}^{(0)}(x)&\alpha_2Q_{n,N}^{(0)}(x)&\cdots&\alpha_{N-1}Q_{n,N}^{(0)}(x)\\
Q_{n,1}^{(0)}(x)&Q_{n,1}(x)&0&\cdots&0\\
Q_{n,2}^{(0)}(x)&0&Q_{n,2}(x)&\cdots&0\\
\vdots&\vdots&\vdots&\ddots&\vdots\\
Q_{n,N-1}^{(0)}(x)&0&0&\cdots&Q_{n,N-1}(x)
\end{pmatrix},\quad n\geq0,
\end{equation}
where $Q_{n,N}(x)$ satisfies the scalar-valued three-term recurrence relation (here $a_{0,N}=\alpha_N$ and $b_{0,N}=\alpha_0$)
\begin{equation}\label{recrel1}
\begin{split}
xQ_{n,N}(x)&=a_{n,N}Q_{n+1,N}(x)+b_{n,N}Q_{n,N}(x)+c_{n,N}Q_{n-1,N}(x),\quad n\geq0,\\
Q_{0,N}(x)&=1,\quad Q_{-1,N}(x)=0,
\end{split}
\end{equation}
and $Q_{n,N}^{(0)}$ will denote the corresponding associated polynomials (or $0$-th associated polynomials). These are polynomials satisfying the same three-term recurrence relation \eqref{recrel1} but with initial conditions $Q_{0,N}^{(0)}=0,Q_{1,N}^{(0)}=-1/\alpha_N$. Also $Q_{n,k}(x),k=1,\ldots,N-1,$ satisfy the scalar-valued three-term recurrence relations
\begin{equation}\label{recrel2}
\begin{split}
xQ_{n,k}(x)&=a_{n+1,k}Q_{n+1,k}(x)+b_{n+1,N}Q_{n,k}(x)+c_{n+1,N}Q_{n-1,k}(x),\quad n\geq0,\\
Q_{0,k}(x)&=1,\quad Q_{-1,k}(x)=0,
\end{split}
\end{equation}
and $Q_{n,k}^{(0)}(x),k=1,\ldots,N-1,$ will denote the corresponding associated polynomials  with initial conditions $Q_{0,k}^{(0)}=0,Q_{1,k}^{(0)}=-c_{1,k}/a_{1,k},k=1,\ldots,N-1.$ Observe that the associated polynomials have degree $n-1$. Therefore the matrix-valued polynomials $(\bm Q_n)_{n\geq0}$ \eqref{polmat} satisfy $\deg(\bm Q_n)=n$ and have nonsingular leading coefficient. 

Let us now see that the discrete-time birth-death chain $\{S_n,n=0,1,\ldots\}$ defined on the spider $\mathbb{S}_N$ and with transition probability matrix $\bm P$ \eqref{tpmm} can be identified with a $N\times N$ \emph{weight matrix} $\bm W$ supported on the real line, i.e. a matrix of measures which is symmetric and nonnegative definite for any Borel set $A$, i.e. $\bm W(A)\geq0$, and with finite moments. The matrix-valued polynomials $(\bm Q_n)_{n\geq0}$ defined by \eqref{polmat} will be orthogonal with respect to $\bm W$ in the following sense
$$
\int_{-1}^1\bm Q_n(x)d\bm W(x)\bm Q_m^T(x)=\|\bm Q_n\|_{\bm W}^2\delta_{nm},
$$
as we will see in the following result.
\begin{proposition}\label{prop1}
Let $\{S_n,n=0,1,\ldots\}$ be a discrete-time birth-death chain on a spider $\mathbb{S}_N$ with transition probability matrix $\bm P$ \eqref{tpmm}. Then there exists a weight matrix $\bm W$ supported on the interval $[-1,1]$ such that the polynomials $(\bm Q_n)_{n\geq0}$ defined by \eqref{polmat} are orthogonal with respect to $\bm W$. 
\end{proposition}
\begin{proof}
For the existence and orthogonality we apply Theorem 2.1 of \cite{DRSZ}. We need to define a sequence of nonsingular matrices $(\bm T_n)_{n\geq0}$ such that
\begin{equation}\label{condss}
\begin{split}
\bm T_n\bm T_n^T\bm B_n&=\bm B_n^T\bm T_n\bm T_n^T,\quad n\geq0,\\
\bm T_n\bm T_n^T\bm A_n&=\bm C_{n+1}^T\bm T_{n+1}\bm T_{n+1}^T,\quad n\geq0,
\end{split}
\end{equation}
where the coefficients $(\bm A_n)_{n\geq0},(\bm B_n)_{n\geq0}$ and $(\bm C_n)_{n\geq1}$ are defined by \eqref{B0A0}, \eqref{ABn} and \eqref{Cn}, respectively. Let us define the following sequences
\begin{equation*}\label{piss}
\begin{split}
\pi_{0,N}&=1,\quad \pi_{n,N}=\alpha_N\frac{a_{1,N}\cdots a_{n-1,N}}{c_{1,N}\cdots c_{n,N}},\quad n\geq1,\\
\pi_{1,k}&=\frac{\alpha_k}{c_{1,k}},\quad \pi_{n+1,k}=\alpha_k\frac{a_{1,k}\cdots a_{n,k}}{c_{1,k}\cdots c_{n+1,k}},\quad n\geq1,\quad k=1,\ldots,N-1.
\end{split}
\end{equation*}
Then a straightforward computation shows that the diagonal matrices 
\begin{equation}\label{Rn}
\bm T_n=\mbox{diag}\left(\sqrt{\pi_{n,N}},\sqrt{\pi_{n+1,1}},\ldots,\sqrt{\pi_{n+1,N-1}}\right),\quad n\geq0,
\end{equation}
satisfy the conditions \eqref{condss}. Finally the weight matrix $\bm W$ is supported on the interval $[-1,1]$ as a consequence of Theorem 2.5 of \cite{DRSZ}.
\end{proof}
\begin{remark}\label{rem22}
From \eqref{Rn} we can define the sequence of matrix-valued \emph{potential coefficients} for the birth-death chain $\{S_n,n=0,1,\ldots\}$ on a spider $\mathbb{S}_N$ as
\begin{equation}\label{Pinn}
\bm \Pi_n=\bm T_n\bm T_n^T=\mbox{diag}\left(\pi_{n,N},\pi_{n+1,1},\ldots,\pi_{n+1,N-1}\right),\quad n\geq0.
\end{equation}
From formula (2.6) of \cite{dI1} we have that this sequence can be identified with the inverse of the norms of the matrix-valued orthogonal polynomials $(\bm Q_n)_{n\geq0}$ defined by \eqref{polmat}, i.e.
\begin{equation*}
\bm\Pi_n=\left(\|\bm Q_n\|_{\bm W}^2\right)^{-1}=\left(\int_{-1}^1\bm Q_n(x)d\bm W(x)\bm Q_n^T(x)\right)^{-1},\quad n\geq0.
\end{equation*}
\end{remark}

\begin{remark}\label{rem23}
The existence of a weight matrix $\bm W$ for the birth-death chain $\{S_n,n=0,1,\ldots\}$ on a spider $\mathbb{S}_N$ gives one way of computing the $(i,j)$-block of the $n$-step transition probability matrix $\bm P^n$ by the so-called \emph{Karlin-McGregor formula}. Indeed, using Theorem 2.3 of \cite{DRSZ} we have that
\begin{equation*}
\bm P_{ij}^n=\left(\int_{-1}^1x^n\bm Q_i(x)d\bm W(x)\bm Q_j^T(x)\right)\bm\Pi_j,\quad i,j\in\mathbb{N}_0,
\end{equation*}
where $(\bm Q_n)_{n\geq0}$ and $(\bm\Pi_n)_{n\geq0}$ are defined by \eqref{polmat} and \eqref{Pinn}, respectively. Also, in \cite{DRSZ}, one can find some probabilistic results concerning \emph{recurrence} or \emph{canonical moments} of quasi-birth-death processes in terms of the spectral matrix $\bm W$. In particular, from Corollary 4.1 of \cite{DRSZ}, we have that the birth-death chain on a spider is recurrent if and only if 
\begin{equation}\label{condrec}
e_j^T\left(\int_{-1}^1\frac{d\bm W(x)}{1-x}\right)\bm T_0^{-1}e_j=\infty, 
\end{equation}
for some $j=1,\ldots,N$, where $e_j=(0,\ldots,1,\ldots,0)^T$ denotes the $j$-th unit vector in $\mathbb{R}^N$. And from Corollary 4.2 of \cite{DRSZ} we have that the birth-death chain on a spider is positive recurrent if and only if one of the measures $e_j^Td\bm W(x)\bm T_0^{-1}e_j, j=1,\ldots, N,$ has a jump at the point 1. 
\end{remark}

The scalar-valued polynomials $(Q_{n,k})_{n\geq0}$, $k=1,\ldots,N,$ and the corresponding associated polynomials $(Q_{n,k}^{(0)})_{n\geq0}$, $k=1,\ldots,N,$ are defined in terms of regular three-term recurrence relations (see \eqref{recrel1} and \eqref{recrel2}). This means, using Favard's theorem or the spectral theorem for orthogonal polynomials, that there exist spectral measures supported on the interval $[-1,1]$ (the corresponding Jacobi matrices are stochastic) such that these polynomials are orthogonal (see \cite{KMc6}). For $k=1,\ldots,N,$ let us denote $\omega_k$ and $\omega_k^{(0)}$ the spectral probability measures associated with the polynomials $(Q_{n,k})_{n\geq0}$ and $(Q_{n,k}^{(0)})_{n\geq0}$, respectively. For any measure $\omega$ supported on the real line let us define the \emph{Stieltjes transform} of $\omega$ by
\begin{equation}\label{Sttrdef}
B(z;\omega)=\int_{\mathbb{R}}\frac{d\omega(x)}{x-z},\quad z\in\mathbb{C}\setminus\mathbb{R}.
\end{equation} 
There is a very well-known connection between the Stieltjes transforms of $\omega_k$ and $\omega_k^{(0)}, k=1,\ldots,N$, which can be found for instance in formula (6) of \cite{KMc6}. These formulas are given by
%\begin{equation}\label{bzomerel}
%B(z;\omega_k)=-\frac{1}{z-b_{1,k}+a_{1,k}c_{2,k}B(z;\omega_k^{(0)})},\; k=1,\ldots,N-1,\; B(z;\omega_N)=-\frac{1}{z-\alpha_0+\alpha_Nc_{1,N}B(z;\omega_N^{(0)})}.
%\end{equation}
\begin{equation}\label{bzomerel}
\begin{split}
B(z;\omega_N)&=-\frac{1}{z-\alpha_0+\alpha_Nc_{1,N}B(z;\omega_N^{(0)})},\\
B(z;\omega_k)&=-\frac{1}{z-b_{1,k}+a_{1,k}c_{2,k}B(z;\omega_k^{(0)})},\quad k=1,\ldots,N-1.
\end{split}
\end{equation}
From Proposition \ref{prop1} we know that any birth-death chain $\{S_n,n=0,1,\ldots\}$ on a spider $\mathbb{S}_N$ can be identified with some weight matrix $\bm W$. Let us give one criterium to compute the Stieltjes transform of $\bm W$ (entry by entry) in terms of the Stieltjes transforms of the measures $\omega_k,k=1,\ldots,N,$ associated with the polynomials $(Q_{n,k})_{n\geq0},k=1,\ldots,N$. For that we will need the following notation
\begin{equation*}\label{aaass}
\vec{\bm\alpha}=\left(\alpha_1,\ldots,\alpha_{N-1}\right)^T, \quad \bm\alpha_D=\mbox{diag}\left(\alpha_1,\alpha_2,\ldots,\alpha_{N-1}\right),
\end{equation*}
\begin{equation}\label{ccss}
\vec{\bm c}=\left(c_{1,1},\ldots,c_{1,N-1}\right)^T, \quad \bm c_D=\mbox{diag}\left(c_{1,1},c_{1,2},\ldots,c_{1,N-1}\right),
\end{equation}
\begin{equation}
\vec{\bm\omega}(x)=\left(\omega_1(x),\ldots,\omega_{N-1}(x)\right)^T,\quad \bm\omega_D(x)=\mbox{diag}\left(\omega_1(x),\omega_2(x),\ldots,\omega_{N-1}(x)\right).
\end{equation}
Whenever we write $B(z;\vec{\bm\omega})$ or $B(z;\bm\omega_D)$ we mean that we are taking the Stieltjes transform on each component/entry.

\begin{proposition}\label{prop2}
Let $\{S_n,n=0,1,\ldots\}$ be a birth-death chain on a spider $\mathbb{S}_N$ with transition probability matrix $\bm P$ \eqref{tpmm}. The Stieltjes transform of the weight matrix $\bm W$ obtained in Proposition \ref{prop1} can be written as
\begin{equation}\label{Sttgen}
B(z;\bm W)=\left(
\begin{array}{c|c}
0&\bm 0\\
\hline
\bm 0&-B(z;\bm\omega_D)\bm c_D\bm\alpha_D^{-1}
\end{array}
\right)+\mathfrak{b}(z)\left(
\begin{array}{c|c}
1&-\vec{\bm c}\,^TB(z;\bm\omega_D)\\
\hline
-B(z;\bm\omega_D)\vec{\bm c}&B(z;\bm\omega_D)\vec{\bm c}\,\vec{\bm c}\,^TB(z;\bm\omega_D)
\end{array}
\right),
\end{equation}
where
\begin{equation}\label{B11z}
\mathfrak{b}(z)=\frac{1}{\D\frac{1}{B(z;\omega_N)}-\vec{\bm\alpha}^TB(z;\bm\omega_D)\vec{\bm c}}.
\end{equation}
\end{proposition}
\begin{proof}
We will use Theorem 2.1 of \cite{Clay}, which gives a relation between the Stieltjes transform of $\bm W$ and the Stieltjes transform of the spectral matrix $\bm W^{(0)}$ of the $0$-th associated process (built from the original one by eliminating the first block row and column of $\bm P$ \eqref{tpmm}). This relation is given by
\begin{equation*}\label{relST}
B(z;\bm W)\bm\Pi_0=-\left[z\bm I_N-\bm B_0+\bm A_0B(z;\bm W^{(0)})\bm\Pi_0^{(0)}\bm C_1\right]^{-1},
\end{equation*}
where $\bm\Pi_0^{(0)}=\bm I_N$ and (see \eqref{Pinn})
\begin{equation}\label{pi0}
\bm \Pi_0=\left(
\begin{array}{c|c}
1&\\
\hline
&\bm\alpha_D\bm c_D^{-1}
\end{array}
\right).
\end{equation}
Since $\bm A_n, \bm B_n, \bm C_{n+1}, n\geq1,$ are diagonal matrices we have that $B(z;\bm W^{(0)})$ is a diagonal matrix given by
$$
B(z;\bm W^{(0)})=\mbox{diag}\left(B(z;\omega_N^{(0)}),B(z;\omega_1^{(0)}),\ldots,B(z;\omega_{N-1}^{(0)})\right).
$$
Using the definition of $\bm B_0,\bm A_0$ and $\bm C_1$ in \eqref{B0A0} and \eqref{Cn} we can write the Stieltjes transform of $\bm W$ in a $2\times2$ block matrix expression
\begin{equation}\label{Bzw}
B(z;\bm W)=-\left(
\begin{array}{c|c}
M_{11}&M_{12}\\
\hline
M_{21}&M_{22}
\end{array}
\right)^{-1}\left(
\begin{array}{c|c}
1&\\
\hline
&\bm c_D\bm\alpha_D^{-1}
\end{array}
\right),
\end{equation}
where (see \eqref{bzomerel})
\begin{align*}
M_{11}&=z-\alpha_0+\alpha_Nc_{1,N}B(z;\omega_N^{(0)})=-\frac{1}{B(z;\omega_N)},\quad M_{12}=-\vec{\bm\alpha}^T,\quad M_{21}=-\vec{\bm c},\\
M_{22}&=\mbox{diag}\left(z-b_{1,1}+a_{1,1}c_{2,1}B(z;\omega_1^{(0)}),\ldots,z-b_{1,N-1}+a_{1,N-1}c_{2,N-1}B(z;\omega_{N-1}^{(0)})\right)=-B(z;\bm\omega_D)^{-1}.
\end{align*}
Using the well-known formula for the inverse of a $2\times2$ block matrix
$$
\left(
\begin{array}{c|c}
A&B\\
\hline
C&D
\end{array}
\right)^{-1}
=\left(
\begin{array}{c|c}
(A-BD^{-1}C)^{-1}&-(A-BD^{-1}C)^{-1}BD^{-1}\\
\hline
-D^{-1}C(A-BD^{-1}C)^{-1}&D^{-1}+D^{-1}C(A-BD^{-1}C)^{-1}BD^{-1}
\end{array}
\right),
$$
in \eqref{Bzw} and the fact that $\vec{\bm\alpha}^T\bm c_D\bm\alpha_D^{-1}=\vec{\bm c}\,^T$, we get \eqref{Sttgen} after some straightforward computations.

\end{proof}

%%%%%%%%%%%%%%%%%%%%%
%%%%%%%%%%%%%%%%%%%%%
%%%%%%%%%%%%%%%%%%%%%
%%%%%%%%%%%%%%%%%%%%%

\section{Reflecting-absorbing factorization for birth-death chains on a spider}\label{sec3}

In this section we decompose the discrete-time birth-death chain $\{S_n,n=0,1,\ldots\}$ on a spider $\mathbb{S}_N$ described by $\bm P$ \eqref{tpmm} into two independent processes: the first one is a \textit{reflecting} process from state 0 and the second one is an \textit{absorbing} process to the state 0. Therefore we are looking for a factorization of the form $\bm P=\bm P_R\bm P_A$, where
\begin{equation}\label{PRA}
\bm P_R=\left(
\begin{array}{ccccc}
\bm Y_0&\bm X_0&\\
&\bm Y_{1}&\bm X_{1}&\\
&&\ddots&\ddots
\end{array}
\right), \quad 
\bm P_A=\left(
\begin{array}{cccc}
\bm S_0&\\
\bm R_{1}&\bm S_{1}&\\
&\ddots&\ddots
\end{array}
\right),
\end{equation}
with blocks given by 
\begin{equation}\label{yyss}
\bm Y_0=\left(
\begin{array}{ccccc}
\beta_0&\beta_1&\beta_2&\cdots&\beta_{N-1}\\
0&y_{1,1}&0&\cdots&0\\
0&0 &y_{1,2}&\cdots&0\\
\vdots&\vdots&\vdots&\ddots&\vdots\\
0&0&0&\cdots&y_{1,N-1}\\
\end{array}
\right),\quad \bm S_0=\left(
\begin{array}{ccccc}
1&0&0&\cdots&0\\
r_{1,1}&s_{1,1}&0&\cdots&0\\
r_{1,2}&0 &s_{1,2}&\cdots&0\\
\vdots&\vdots&\vdots&\ddots&\vdots\\
r_{1,N-1}&0&0&\cdots&s_{1,N-1}\\
\end{array}
\right),
\end{equation}
\begin{equation}\label{ssss}
\bm Y_n=\mbox{diag}\left(y_{n,N}, y_{n+1,1},\ldots, y_{n+1,N-1}\right),\quad
\bm S_n=\mbox{diag}\left(s_{n,N}, s_{n+1,1},\ldots, s_{n+1,N-1}\right),\quad n\geq1,
\end{equation}
\begin{equation*}\label{xxss}
\bm X_0=\mbox{diag}\left(\beta_N, x_{1,1},\ldots, x_{1,N-1}\right),\quad 
\bm X_n=\mbox{diag}\left(x_{n,N}, x_{n+1,1},\ldots, x_{n+1,N-1}\right),\quad n\geq1,
\end{equation*}
\begin{equation}\label{rrss}
\bm R_n=\mbox{diag}\left(r_{n,N}, r_{n+1,1},\ldots, r_{n+1,N-1}\right),\quad n\geq1,
\end{equation}
with the conditions that all these matrices are stochastic, i.e. $\sum_{k=0}^{N}\beta_k=1$ and
\begin{align}
x_{n,m}+y_{n,m}&=1, \quad n\ge 1,\quad m=1, 2, \dots, N, \label{xy}\\
r_{n,m}+s_{n,m}&=1, \quad n\ge 1, \quad m=1, 2. \dots, N.\label{rs}
\end{align}
Diagrams of the possible transitions between the states of both birth-death chains are given in Figure \ref{fig2}.
\begin{figure}[h]
\includegraphics[scale=0.25]{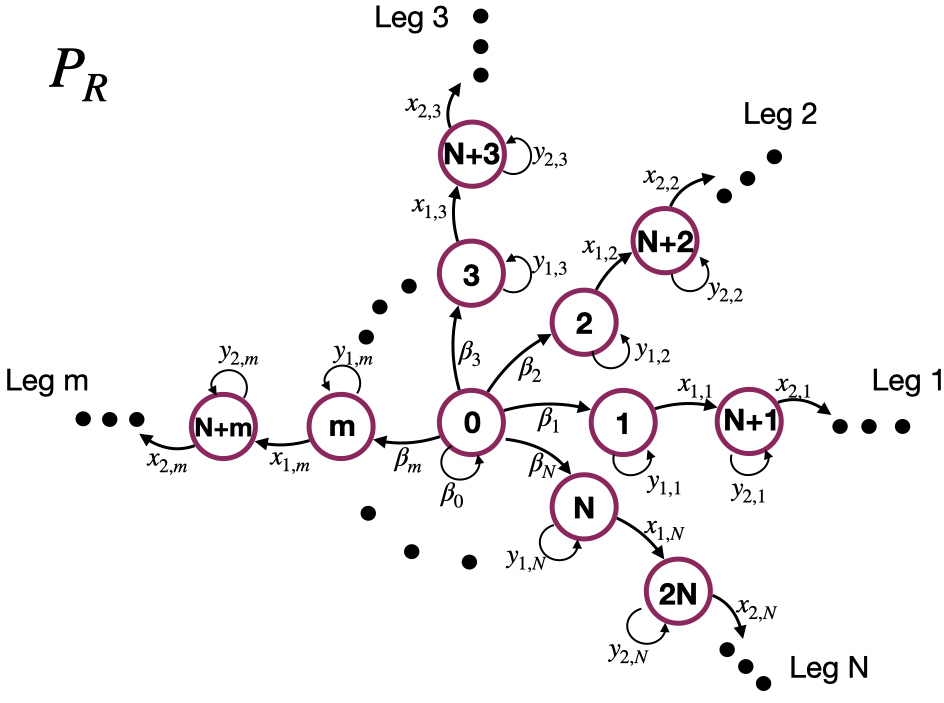}
\includegraphics[scale=0.25]{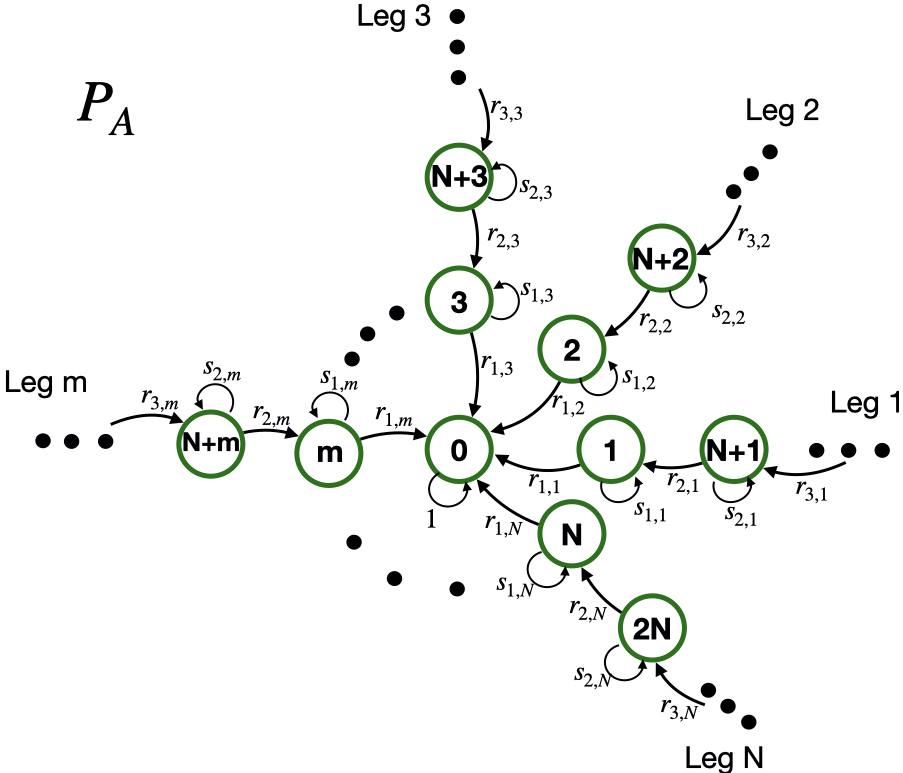}
\caption{Diagrams of the possible transitions between the states of the reflecting process given by $\bm P_R$ and the absorbing process given by $\bm P_A$.}
\label{fig2}
\end{figure}
Observe that the reflecting-absorbing factorization $\bm P=\bm P_R\bm P_A$ is just a particular case of a block UL factorization (but not the only one). A simple computation from $\bm P=\bm P_R\bm P_A$ gives the following relations
\begin{equation}\label{ABCXYRS}
\begin{split}
\bm A_n&=\bm X_n\bm S_{n+1}, \quad n\ge 0,\\
\bm B_n&=\bm X_n\bm R_{n+1}+\bm Y_n\bm S_n, \quad n\ge 0,\\ 
\bm C_n&=\bm Y_n\bm R_n, \quad n\ge 1.
\end{split}
\end{equation}
If we look entry by entry these relations can also be written as
\begin{align}
\alpha_0&=\beta_0+\sum_{k=1}^{N}\beta_k r_{1,k}, \nonumber\\
\alpha_m&=\beta_m s_{1,m}, \quad m=1, 2, \dots, N,\label{alpha11}\\
a_{n,m}&=x_{n,m}s_{n+1,m},\quad n\ge 1, \quad  m=1, 2, \dots, N, \label{aaD}\\
b_{n,m}&=y_{n,m}s_{n,m}+x_{n,m}r_{n+1,m},\quad n\ge 1, \quad  m=1, 2, \dots, N, 
\nonumber\\
c_{n,m}&=y_{n,m}r_{n,m}, \quad n\ge 1, \quad m=1, 2, \dots, N.\label{ccD}
\end{align}
We can compute all the coefficients $x_{n,m},y_{n,m},r_{n,m},s_{n,m},$ in terms of $N$ free parameters $\beta_1,\ldots,\beta_N$, one for each leg. Indeed, if we fix $\beta_m$ for $m = 1, 2, \dots, N,$ we get $s_{1, m}$, for $m= 1, 2, \dots, N,$ from equation \eqref{alpha11} and $r_{1, m}$, for $m= 1, 2, \dots, N,$ from equation \eqref{rs}. After this we get $y_{1, m}$, for $m= 1, 2, \dots, N,$ from equation \eqref{ccD} and $x_{1, m}$, for $m= 1, 2, \dots, N,$ from equation \eqref{xy}. Then we get $s_{2, m}$, for $m= 1, 2, \dots, N$ from equation \eqref{aaD} and so on using the same equations. %Alternatively, we can compute first $\bm S_0,\bm Y_0,\bm X_0,\bm S_1,\bm R_1,\bm Y_1,\bm X_1,$ from the description given above and then, using the fact that $\bm S_n,\bm R_n,\bm X_n,\bm Y_n,n\geq2,$ are diagonal matrices, we can get these coefficients recurrently (see the first and third equation in \eqref{ABCXYRS}) as
%\begin{align*}
%\bm S_{n+1}&=\left(\bm I_N-\bm C_n\left(\bm I_N-\bm S_n\right)^{-1}\right)^{-1}\bm A_n,\quad \bm R_{n+1}=\bm I_N-\bm S_{n+1},\quad n\geq1,\\
%\bm Y_{n+1}&=\bm C_{n+1}\left(\bm I_N-\left(\bm I_N-\bm Y_n\right)^{-1}\bm A_n\right)^{-1},\quad \bm X_{n+1}=\bm I_N-\bm X_{n+1},\quad n\geq1.
%\end{align*}

In the same fashion as in \cite{GdI3, dIJ1, dIJ2}, let us now see under what conditions we can guarantee a stochastic reflecting-absorbing factorization. Let 
\begin{equation}\label{Hcf}
H_m=\cFrac{\alpha_m}{1}-\cFrac{c_{1,m}}{1}-\cFrac{a_{1,m}}{1}-\cFrac{c_{2,m}}{1}-\cdots,\quad m=1, 2, \ldots, N,
\end{equation}
be the continued fraction with sequence of convergents given by
\begin{equation}\label{conv}
h_{n,m}=\frac{A_{n,m}}{B_{n,m}},\quad n\geq0,\quad m=1,\ldots,N.
\end{equation}
The sequences $(A_{n,m})_{n\geq0}$ and $(B_{n,m})_{n\geq0}$ for every  $m=1,\ldots,N,$ can be recursively obtained using the formulas
\begin{equation*}
\begin{split}
A_{2n,m}&=A_{2n-1,m}-c_{n,m}A_{2n-2,m},\quad n\geq1,\quad A_{2n+1,m}=A_{2n,m}-a_{n,m}A_{2n-1,m},\quad n\geq0,\\
A_{-1,m}&=-1, \quad A_{0,m}=0,
\end{split}
\end{equation*}
\begin{equation*}
\begin{split}
B_{2n,m}&=B_{2n-1,m}-c_{n,m}B_{2n-2,m},\quad n\geq1,\quad B_{2n+1,m}=B_{2n,m}-a_{n,m}B_{2n-1,m},\quad n\geq0,\\
B_{-1,m}&=0, \quad B_{0,m}=1,
\end{split}
\end{equation*}
where here we are calling $a_{0,m}=\alpha_m$.

\begin{theorem}\label{thm1}
Let $H_m$, $m=1, 2, \dots, N$, be the continued fractions defined by \eqref{Hcf} with their corresponding sequences of convergents \eqref{conv}. Assume that 
$$0<A_{n,m}<B_{n,m},\quad n\geq0,\quad m=1,\ldots,N.$$
Then the continued fractions $H_m$, $m=1, 2, \dots, N,$ are all convergent. Moreover, let $\bm P=\bm P_R\bm P_A$ and assume that $\sum_{m=1}^{N}H_m<1$. Then $\bm P_R$ and $\bm P_A$ are stochastic matrices if and only if 
\begin{equation}\label{betarange}
\beta_m\ge H_m,\quad m=1, 2, \dots, N.
\end{equation}
\end{theorem}
\begin{proof}
For $m= 1, 2, \dots, N$, it is not hard to proof that 
\begin{equation*}\label{diff}
\begin{split}
A_{2n,m}B_{2n+1,m}-B_{2n,n}A_{2n+1,m}&=-\alpha_mc_{1,m}a_{1,m}\cdots a_{n,m},\quad n\geq0,\\
A_{2n+1,m}B_{2n+2,m}-B_{2n+1,m}A_{2n+2,m}&=-\alpha_mc_{1,m}a_{1,m}\cdots c_{n+1,m},\quad n\geq0,
\end{split}
\end{equation*}
and consequently
\begin{equation*}
\begin{split}
h_{2n,m}-h_{2n+1,m}&=\frac{A_{2n,m}}{B_{2n,m}}-\frac{A_{2n+1,m}}{B_{2n+1,m}}=-\frac{\alpha_mc_{1,m}a_{1,m}\cdots a_{n,m}}{B_{2n,m}B_{2n+1,m}}<0, \quad n\ge 0,\\
h_{2n+1,m}-h_{2n+2,m}&=\frac{A_{2n+1,m}}{B_{2n+1,m}}-\frac{A_{2n+2,m}}{B_{2n+2,m}}=-\frac{\alpha_mc_{1,m}a_{1,m}\cdots c_{n+1,m}}{B_{2n+2,m}B_{2n+1,m}}<0, \quad n\ge 0.
\end{split}
\end{equation*}
Therefore we conclude that 
$$
0=h_{0,m}<\cdots <h_{2n,m}<h_{2n+1,m}<h_{2n+2,m}< \cdots <1,
$$
and then the sequences $(h_{n,m})_{n\geq0}$ are all bounded and strictly increasing, so they converge to $H_m$ for every $m=1,\ldots,N$. Now assume that $\sum_{m=1}^{N}H_m<1$ and $\bm P_R$ and $\bm P_A$ are stochastic matrices. Then it is clear that 
$$
\beta_m>0=h_{0,m},
$$
and using equation \eqref{alpha11} we have 
$$
s_{1,m}=\frac{\alpha_m}{\beta_m}<1 \Leftrightarrow \beta_m >\alpha_m=h_{1,m}.
$$
Using now equations \eqref{ccD}, \eqref{rs} and \eqref{alpha11} we have 
$$
y_{1,m}=\frac{c_{1,m}}{r_{1,m}}  \Leftrightarrow  1-s_{1,m}>c_{1,m}  \Leftrightarrow \beta_m>\frac{\alpha_m}{1-c_{1,m}}=h_{2,m},
$$
and in general it can be shown that 
$$
\beta_m>h_{n,m}.
$$
Therefore  $0=h_{0,m}<h_{n,m}<H_m\le \beta_m$. On the contrary, if \eqref{betarange} holds, in particular we have that $\beta_m>h_{n,m}$ for every $n\geq0, m=1,\ldots,N$. Following the same steps as before, using an argument of strong induction, will lead us to the fact that both $\bm P_R$ and $\bm P_L$ are stochastic matrices (see the proof of Proposition 2.1 of \cite{dIJ2} for more details).
\end{proof}

\begin{remark}
The reflecting-absorbing factorization $\bm P=\bm P_R\bm P_A$ is just one type of a stochastic block UL factorization of $\bm P$, but there can be more possibilities. Also we could have considered a stochastic block LU factorization of $\bm P$. As it was pointed out in \cite{GdI4}, the different stochastic block factorizations of $\bm P$ may come with many degrees of freedom, and the analysis is more complicated than the case of classical birth-death chains.
\end{remark}

\subsection{Stochastic Darboux transformation and the associated spectral matrix}\label{ssDarboux}

Once we have the conditions under we can perform a stochastic reflecting-absorbing factorization, it is possible to compute what is called a \emph{discrete Darboux transformation}, consisting on inverting the order of the factors. The Darboux transformation has a long history but probably the first reference of a discrete Darboux transformation like we study here appeared in \cite{MS} in connection with the Toda lattice.

If $\bm P=\bm P_R\bm P_A$ as in \eqref{PRA}, then, by inverting the order of the factors, we obtain another stochastic matrix of the form $\widetilde{\bm P}=\bm P_A\bm P_R$, since the multiplication of two stochastic matrices is again a stochastic matrix. This new matrix preserves the block tridiagonal structure. Also $\widetilde{\bm P}$ will be a family (depending on $N$ free parameters $\beta_1,\ldots,\beta_N$) of Markov chains $\{\widetilde S_n,n=0,1,\ldots\}$ on a spider $\mathbb{S}_N$ which is ``almost'' a family of birth-death chains. The only difference is that we will have extra transitions between the first states of each leg or, in other words, between the states $1,2,\ldots,N$. If we call $\widetilde{\bm B}_n,\widetilde{\bm A}_n,\widetilde{\bm C}_{n+1},n\geq0,$ the new coefficients of the block tridiagonal matrix $\widetilde{\bm P}$, a direct computations shows
\begin{equation}\label{ABCg}
\begin{split}
\widetilde{\bm A}_n&=\bm S_n\bm X_n, \quad n\ge 0,\\
\widetilde{\bm B}_0&=\bm S_0\bm Y_0, \quad \widetilde{\bm B}_n=\bm R_n\bm X_{n-1}+\bm S_n\bm Y_n, \quad n\ge 1, \\
\widetilde{\bm C}_n&=\bm R_n\bm Y_{n-1}, \quad n\ge 1.
\end{split}
\end{equation}
If we look entry by entry these relations can also be written as
\begin{equation*}
\begin{split}
\tilde \alpha_m&=\beta_m, \quad m= 0, 1, \dots, N,\\
\tilde \alpha_{n,m}&=s_{n,m}x_{n,m}, \quad n\geq1, \quad m=1, \dots, N,\\
\tilde b_{1,m}&=r_{1,m}\beta_m+s_{1,m}y_{1,m},\quad m=1, \dots, N, \\
\tilde b_{n,m}&=r_{n,m}x_{n-1,m}+s_{n,m}y_{n,m},\quad n\geq1, \quad m=1, \dots, N,\\
\tilde c_{1,m}&=r_{1,m}\beta_m,\quad m= 0, 1, \dots, N,\\
\tilde c_{n,m}&=r_{n,m}y_{n-1,m}, \quad n\geq1, \quad m=1, \dots, N.\\
\end{split}
\end{equation*}
We also have extra transition probabilities between the first states of each leg, given by 
$$
\tilde d_{i,j}=\beta_jr_{1,i}, \quad i,j =1, \dots, N,\quad i\neq j.
$$
A diagram of this process is similar that the one for the process described by $\bm P$ in Figure \ref{fig1} but now we have to add probabilities between the first states of each leg. For instance, for $N=2$ we get the diagram in page 10 of \cite{dIJ2}. For $N=3$ and $N=4$ we have the diagrams in Figure \ref{fig3}.
\begin{figure}[h]
\includegraphics[scale=0.25]{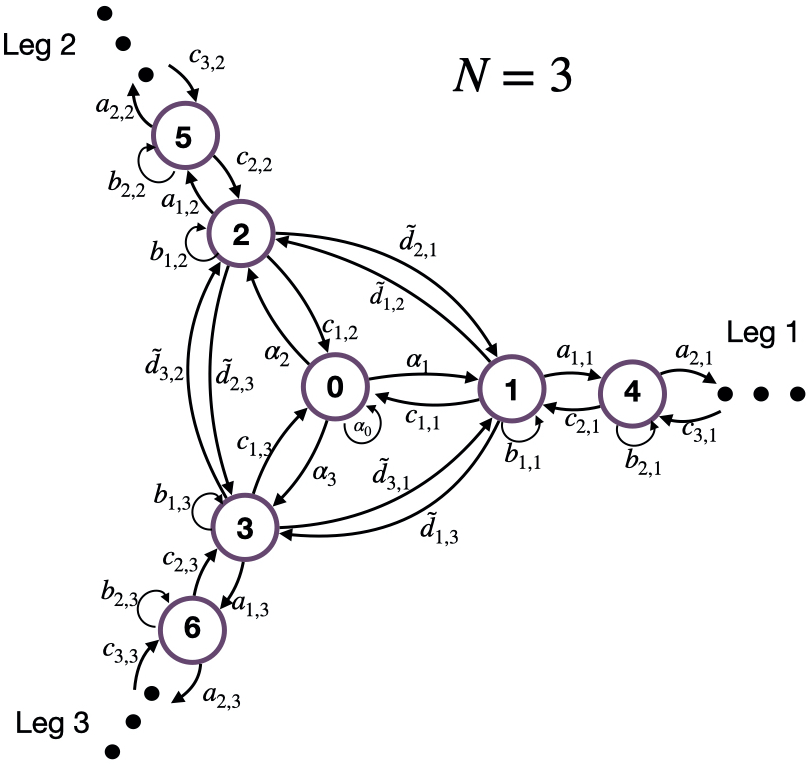}
\includegraphics[scale=0.25]{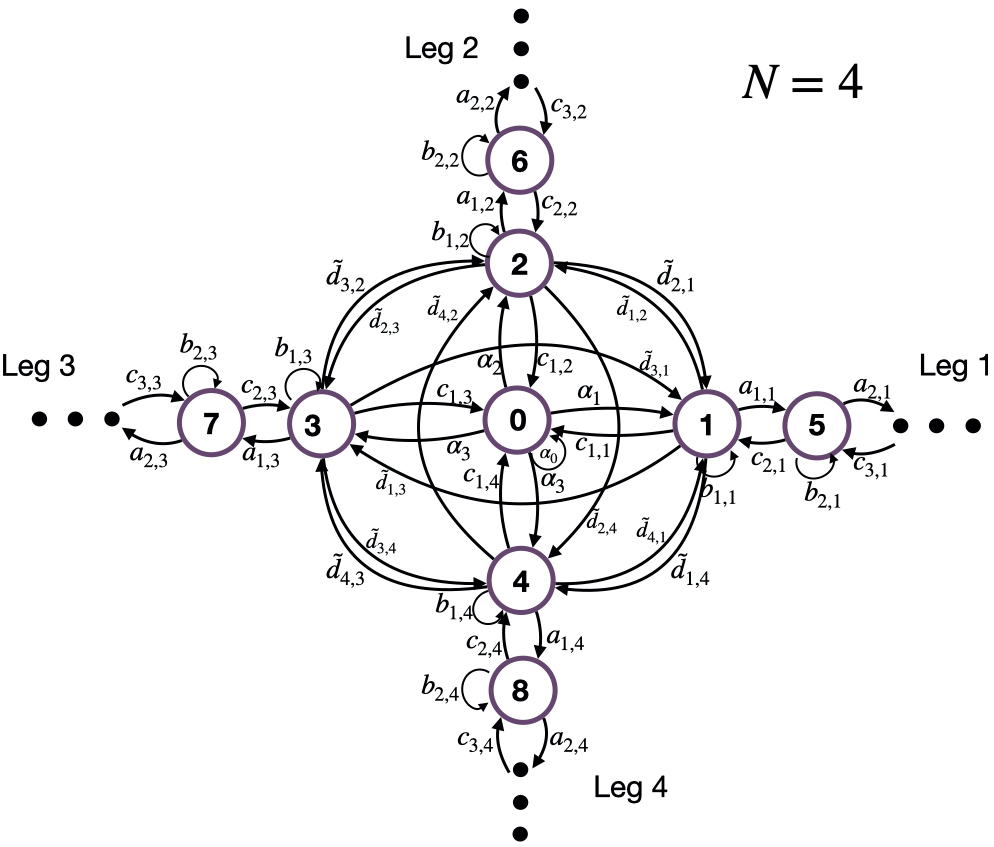}
\caption{Diagrams of the process associated with $\widetilde{\bm P}$ for $N=3$ and $N=4$.}
\label{fig3}
\end{figure}
In general we have to add $N(N-1)$ extra transition probabilities between the first states of each leg. 

In Proposition \ref{prop1} we proved that all birth-death chains $\{S_n, n=0,1,\ldots\}$  on a spider $\mathbb{S}_N$ have an associated weight matrix $\bm W$ supported on the interval $[-1,1]$. Now we wonder if we can find the spectral matrix $\widetilde{\bm W}$ associated with the Darboux transformation $\widetilde{\bm P}$. This is possible using Remark 2.4 of \cite{dIJ2}. The spectral matrix is then given by
\begin{equation}\label{spdtr}
\widetilde{\bm W}(x)=\bm S_0\left(\frac{\bm W(x)}{x}+\left[(\bm\Pi_0\bm Y_0\bm S_0)^{-1}-\bm M_{-1}\right]\delta_0(x)\right)\bm S_0^T,
\end{equation}
where $\bm S_0$ and  $\bm Y_0$ are given by \eqref{yyss}, $\bm \Pi_0$ is given by \eqref{Pinn} (see also \eqref{pi0}), $\bm M_{-1}=\int_{-1}^1x^{-1}\bm W(x)dx$ and $\delta_0(x)$ is the Dirac delta at $0$. A direct computation shows that 
$$
\bm\Pi_0\bm Y_0\bm S_0=\begin{pmatrix}
\alpha_0+\alpha_N-\beta_N&\alpha_1&\alpha_2&\cdots&\alpha_{N-1}\\
\alpha_1&\D\frac{\alpha_1^2}{\beta_1-\alpha_1}&0&\cdots&0\\
\alpha_2&0&\D\frac{\alpha_2^2}{\beta_2-\alpha_2}&\cdots&0\\
\vdots&\vdots&\vdots&\ddots&\vdots\\
\alpha_{N-1}&0&0&\cdots&\D\frac{\alpha_{N-1}^2}{\beta_{N-1}-\alpha_{N-1}}
\end{pmatrix}.
$$
Observe that this matrix is always symmetric. If we call $\bm X=(\bm \Pi_0\bm Y_0\bm S_0)^{-1}$, it can be shown that the entries of $\bm X=(X_{ij})$ for $i\leq j$ (for $i\geq j$ we have to change $i$ by $j$ since $\bm X$ is symmetric) are given by
\begin{equation}\label{XXij}
X_{ij}=\frac{1}{\beta_0}\begin{cases}
1,&\mbox{if}\quad i=j=1,\\
1-\D\frac{\beta_{j-1}}{\alpha_{j-1}},&\mbox{if}\quad i=1,j>1,\\
\left(1-\D\frac{\beta_{i-1}}{\alpha_{i-1}}\right)\left(1-\D\frac{\beta_{j-1}}{\alpha_{j-1}}\right),&\mbox{if}\quad i>1,j>i,\\
\left(1-\D\frac{\beta_{i-1}}{\alpha_{i-1}}\right)\left(1-\D\frac{\beta_{0}}{\alpha_{i-1}}-\D\frac{\beta_{i-1}}{\alpha_{i-1}}\right),&\mbox{if}\quad i>1,i=j.
\end{cases}
\end{equation}
However, it may be hard to compute the moment $\bm M_{-1}$. But observe that $\bm M_{-1}=B(0;\bm W)$, where $B(z;\bm W)$ is the Stieltjes transform of $\bm W$ defined by \eqref{Sttrdef}. Using Proposition \ref{prop2} we may have a way to compute explicitly $B(z;\bm W)$. In some cases, as we will see in the example of the next section, if will be possible to compute explicitly $B(z;\bm W)$ and therefore the moment $\bm M_{-1}$.

Finally, we can also compute the matrix-valued orthogonal polynomials $(\widetilde{\bm Q}_n)_{n\geq0}$ associated with $\widetilde{\bm W}$ using Theorem 2.3 of \cite{dIJ2}. Consider the matrix-valued polynomials
\begin{equation*}
\bm U_0(x)=\bm S_0\bm Q_0(x)=\bm S_0,\quad \bm U_n(x)=\bm R_n\bm Q_{n-1}(x)+\bm S_n\bm Q_n(x),\quad n\geq1,
\end{equation*}
where $(\bm S_n)_{n\geq0}$ and $(\bm R_n)_{n\geq1}$ are defined by \eqref{yyss}, \eqref{ssss} and \eqref{rrss}, respectively. If we denote by $\bm Q=\left(\bm Q_0^T,\bm Q_1^T,\cdots\right)^T$ and $\bm U=\left(\bm U_0^T,\bm U_1^T,\cdots\right)^T$, then we have that $\bm U=\bm P_A\bm Q$, where $\bm P_A$ is given by \eqref{PRA}. The matrix-valued orthogonal polynomials $(\widetilde{\bm Q}_n)_{n\geq0}$  are then defined by 
$$
\widetilde{\bm Q}_n(x)=\bm U_n(x)\bm S_0^{-1},\quad n\geq0.
$$
Since we have an explicit expression of the polynomials $(\bm Q_n)_{n\geq0}$ in \eqref{polmat} and
\begin{equation}\label{ssinv}
\bm S_0^{-1}=\left(
\begin{array}{ccccc}
1&0&0&\cdots&0\\
-\D\frac{r_{1,1}}{s_{1,1}}&\D\frac{1}{s_{1,1}}&0&\cdots&0\\
-\D\frac{r_{1,2}}{s_{1,2}}&0 &\D\frac{1}{s_{1,2}}&\cdots&0\\
\vdots&\vdots&&\ddots&\vdots\\
-\D\frac{r_{1,N-1}}{s_{1,N-1}}&0&\cdots&&\D\frac{1}{s_{1,N-1}}\\
\end{array}
\right),
\end{equation}
then we have that
\begin{equation}\label{polmatdar}
\widetilde{\bm Q}_n(x)=\begin{pmatrix}
R_{n,N}(x)&\D\frac{\alpha_1}{s_{1,1}}R_{n,N}^{(0)}(x)&\D\frac{\alpha_2}{s_{1,2}}R_{n,N}^{(0)}(x)&\cdots&\D\frac{\alpha_{N-1}}{s_{1,N-1}}R_{n,N}^{(0)}(x)\\
R_{n,1}^{(0)}(x)&R_{n,1}(x)&0&\cdots&0\\
R_{n,2}^{(0)}(x)&0&R_{n,2}(x)&\cdots&0\\
\vdots&\vdots&\vdots&\ddots&\vdots\\
R_{n,N-1}^{(0)}(x)&0&0&\cdots&R_{n,N-1}(x)
\end{pmatrix},\quad n\geq0,
\end{equation}
where
\begin{align*}
R_{n,N}(x)=&s_{n,N}Q_{n,N}(x)+r_{n,N}Q_{n-1,N}(x)-\left(s_{n,N}Q_{n,N}^{(0)}(x)+r_{n,N}Q_{n-1,N}^{(0)}(x)\right)\sum_{k=1}^{N-1}\frac{r_{1,k}\alpha_k}{s_{1,k}},\\
R_{n,N}^{(0)}(x)=&s_{n,N}Q_{n,N}^{(0)}(x)+r_{n,N}Q_{n-1,N}^{(0)}(x),\\
R_{n,k}(x)=&\frac{1}{s_{1,k}}\left(s_{n+1,k}Q_{n,k}(x)+r_{n+1,k}Q_{n-1,k}(x)\right),\quad k=1,\ldots,N-1,\\
R_{n,k}^{(0)}(x)=&s_{n+1,k}Q_{n,k}^{(0)}(x)+r_{n+1,k}Q_{n-1,k}^{(0)}(x)-\frac{r_{1,k}}{s_{1,k}}\left(s_{n+1,k}Q_{n,k}(x)+r_{n+1,k}Q_{n-1,k}(x)\right).
\end{align*}
As for the potential coefficients (see Remark \ref{rem22}) for the Markov chain $\{\widetilde S_n, n=0,1,\ldots\}$ on a spider $\mathbb{S}_N$ generated by $\widetilde{\bm P}$, we have, using formula (2.41) of \cite{dIJ2}, that
$$
\widetilde{\bm\Pi}_n=\bm Y_n^T\bm\Pi_n\bm S_n^{-1},\quad n\geq0.
$$
Therefore, using \eqref{yyss}, \eqref{ssinv}, \eqref{alpha11} and \eqref{ccD}, we obtain
\begin{equation*}
\widetilde{\bm\Pi}_0=\mbox{diag}\left(\beta_0, \D\frac{\beta_1}{r_{1,1}},\ldots, \D\frac{\beta_{N-1}}{r_{1,N-1}}\right)=\mbox{diag}\left(\beta_0, \D\frac{\beta_1^2}{\beta_1-\alpha_1},\ldots, \D\frac{\beta_{N-1}^2}{\beta_{N-1}-\alpha_{N-1}}\right).
\end{equation*}
Therefore $(\widetilde{\bm\Pi}_n)_{n\geq0}$ are always \emph{diagonal matrices}. As a consequence we obtain the (diagonal) norms of the matrix-valued orthogonal polynomials $(\widetilde{\bm Q}_n)_{n\geq0}$:
\begin{equation*}
\widetilde{\bm\Pi}_n=\left(\|\widetilde{\bm Q}_n\|_{\widetilde{\bm W}}^2\right)^{-1}=\left(\int_{-1}^1\widetilde{\bm Q}_n(x)d\widetilde{\bm W}(x)\widetilde{\bm Q}_n^T(x)\right)^{-1},
\end{equation*}
and the orthogonality relations
\begin{equation*}
\int_{-1}^1\widetilde{\bm Q}_n(x)d\widetilde{\bm W}(x)\widetilde{\bm Q}_m^T(x)=\widetilde{\bm\Pi}_n^{-1}\delta_{n,m}.
\end{equation*}

%%%%%%%%%%%%%%%%%%%%%
%%%%%%%%%%%%%%%%%%%%%
%%%%%%%%%%%%%%%%%%%%%
%%%%%%%%%%%%%%%%%%%%%
\section{An example: random walk on a spider}
\label{sec4}

Consider the block tridiagonal transition probability matrix \eqref{tpmm} with constant transition probability coefficients, i.e.
$$
\bm B_0=\begin{pmatrix}
\alpha_0&\alpha_1&\alpha_2&\cdots&\alpha_{N-1}\\
c&b&0&\cdots&0\\
c&0&b&\cdots&0\\
\vdots&\vdots&\vdots&\ddots&\vdots\\
c&0&0&\cdots&b
\end{pmatrix},\; \bm A_0=\mbox{diag}\left(\alpha_N, a,\ldots, a\right),\;
\bm A_n=a\bm I_N,\; \bm B_n=b\bm I_N,\;\bm C_n=c\bm I_N,\; n\geq1,
$$
where $\sum_{k=0}^N\alpha_k=1,$ and $a+b+c=1.$ Observe that in this case the vector $\vec{\bm c}$ defined by \eqref{ccss} is given by $\vec{\bm c}=c\,\vec{\bm e}_{N-1}$, where  $\vec{\bm e}_k=(1,1,\ldots,1)^T$ is the vector of dimension $k$ with all components equal to 1. Also the matrix $\bm c_D$ defined by \eqref{ccss} is given by $\bm c_D=c\,\bm I_{N-1}$. The Stieltjes transforms of $\omega_k^{(0)},k=1,\ldots,N,$ do not depend on $k$ and are given by \begin{equation}\label{Bzw0}
B(z;\omega_k^{(0)})=\frac{b-z+\sqrt{(z-\sigma_+)(z-\sigma_-)}}{2ac},\quad \sigma_{\pm}=1-(\sqrt{a}\mp\sqrt{c})^2.
\end{equation}
Therefore, using the second formula in \eqref{bzomerel} and rationalizing, we obtain
$$
B(z;\omega_k)=-\frac{2}{z-b+\sqrt{(z-\sigma_+)(z-\sigma_-)}}=\frac{b-z+\sqrt{(z-\sigma_+)(z-\sigma_-)}}{2ac},\quad k=1,\ldots,N-1.
$$
As a consequence
$$
B(z;\bm\omega_D)=\frac{b-z+\sqrt{(z-\sigma_+)(z-\sigma_-)}}{2ab}\bm I_{N-1}.
$$
On the other hand, using the first formula in \eqref{bzomerel} and \eqref{Bzw0}, we obtain
$$
\frac{1}{B(z;\omega_N)}=-\left[z-\alpha_0+\alpha_NcB(z;\omega_N^{(0)})\right]=-\left[z-\alpha_0+\frac{\alpha_N}{2a}\left(b-z+\sqrt{(z-\sigma_+)(z-\sigma_-)}\right)\right].
$$
Using the previous two formulas in \eqref{B11z} and the fact that $\alpha_1+\cdots+\alpha_{N-1}=1-\alpha_0-\alpha_N$, we obtain an expression for $\mathfrak{b}(z)$ in Proposition \ref{prop2}, given by
$$
\mathfrak{b}(z)=\frac{1}{\alpha_0-\D\frac{b(1-\alpha_0)}{2a}-\left(1-\frac{1-\alpha_0}{2a}\right)z-\frac{1-\alpha_0}{2a}\sqrt{(z-\sigma_+)(z-\sigma_-)}}.
$$
After rationalizing we get
\begin{equation*}\label{betazex}
\mathfrak{b}(z)=\frac{(1-2a-\alpha_0)z-b+\alpha_0(1+a-c)+(1-\alpha_0)\sqrt{(z-\sigma_+)(z-\sigma_-)}}{2(1-z)\left[(1-a-\alpha_0)z+c-\alpha_0(1-a+c-\alpha_0)\right]}.
\end{equation*}
Therefore we have all the functions necessary to compute the Stieltjes transform of $\bm W$, given by \eqref{Sttgen} of Proposition \ref{prop2}. After some computations we can write $B(z;\bm W)$ as
\begin{equation}\label{sttrW}
B(z;\bm W)=\left(
\begin{array}{c|c}
B_{11}(z;\bm W)&B_{12}(z;\bm W)\vec{\bm e}\,^T_{N-1}\\
\hline
B_{12}(z;\bm W)\vec{\bm e}_{N-1}&B_{22}(z;\bm W)
\end{array}
\right),
\end{equation}
where
\begin{align*}
B_{11}(z;\bm W)&=\mathfrak{b}(z),\\
B_{12}(z;\bm W)&=\frac{2c-\alpha_0(1-a+c)+(b+\alpha_0)z-z^2+(z-\alpha_0)\sqrt{(z-\sigma_+)(z-\sigma_-)}}{2(1-z)\left[(1-a-\alpha_0)z+c-\alpha_0(1-a+c-\alpha_0)\right]},\\
B_{22}(z;\bm W)&=\frac{b-z+\sqrt{(z-\sigma_+)(z-\sigma_-)}}{2a}\bm\alpha_D^{-1}\\
&\qquad+\frac{p(z)+r(z)\sqrt{(z-\sigma_+)(z-\sigma_-)}}{2a(1-z)\left[(1-a-\alpha_0)z+c-\alpha_0(1-a+c-\alpha_0)\right]}\vec{\bm e}_{N-1}\vec{\bm e}_{N-1}^T,
\end{align*}
and
\begin{align*}
p(z)&=-z^3+(\alpha_0+2b)z^2-(\alpha_0(2-2a-c)+b^2-2ac-c)z-bc+\alpha_0(b-a(1-a+c)),\\
r(z)&=z^2-(\alpha_0+b)z-c+\alpha_0(1-a).
\end{align*}
The weight matrix $\bm W$ will consist in the addition of an absolutely continuous part $\bm W_c$ and a discrete part $\bm W_d$, i.e. $\bm W=\bm W_c+\bm W_d$. On one hand, using the Perron-Stieltjes inversion formula, we get that $\bm W_c$ is given by
\begin{equation}\label{WWc}
\bm W_c(x)=\left(
\begin{array}{c|c}
W_{11}(x)&W_{12}(x)\vec{\bm e}\,^T_{N-1}\\
\hline
W_{12}(x)\vec{\bm e}_{N-1}&W_{22}(x)
\end{array}
\right),\quad x\in[\sigma_-,\sigma_+],
\end{equation}
where
\begin{align*}
W_{11}(x)&=\frac{(1-\alpha_0)\sqrt{(\sigma_+-x)(x-\sigma_-)}}{2\pi(1-x)\left[(1-a-\alpha_0)x+c-\alpha_0(1-a+c-\alpha_0)\right]},\\
W_{12}(x)&=\frac{(x-\alpha_0)\sqrt{(\sigma_+-x)(x-\sigma_-)}}{2\pi(1-x)\left[(1-a-\alpha_0)x+c-\alpha_0(1-a+c-\alpha_0)\right]},\\
W_{22}(x)&=\frac{\sqrt{(\sigma_+-x)(x-\sigma_-)}}{2\pi a}\bm\alpha_D^{-1}+\frac{r(x)\sqrt{(\sigma_+-x)(x-\sigma_-)}}{2\pi a(1-x)\left[(1-a-\alpha_0)x+c-\alpha_0(1-a+c-\alpha_0)\right]}\bm e_{N-1}\bm e_{N-1}^T.
\end{align*}
On the other hand, the discrete part $\bm W_d$ will be given by the behavior of the Stieltjes transform $B(z;\bm W)$ \eqref{sttrW} at its poles, given in this case by
$$
z_1=1,\quad z_2=\frac{\alpha_0(1-a+c-\alpha_0)-c}{1-a-\alpha_0}.
$$
After some computations we get
\begin{equation}\label{WWd}
\bm W_d(x)=\frac{c-a}{c-a+1-\alpha_0}\delta_1(x)\vec{\bm e}_N\vec{\bm e}\,^T_N\chi_{\{c>a\}}+\frac{(1-\alpha_0-a)^2-ac}{(1-\alpha_0)(1-\alpha_0-a+c)}\delta_{z_2}(x)\vec{\bm u}_N\vec{\bm u}_N^T\chi_{\{(1-\alpha_0-a)^2>ac\}},
\end{equation}
where $\chi_A$ is the indicator function and
$$
\vec{\bm u}_N=\left(-1,\frac{c}{1-\alpha_0-a},\ldots,\frac{c}{1-\alpha_0-a}\right)^T.
$$
From Remark \ref{rem23} we can study recurrence for the random walk on a spider. There will be three cases:
\begin{itemize}
\item If $a>c,$ then we have that $[\sigma_-,\sigma_+]\subsetneq[-1,1]$. Therefore all integrals in \eqref{condrec} are bounded and the random walk on a spider is transient.
\item If $a=c,$ then we have $[\sigma_-,\sigma_+]=[1-4a,1]$. Therefore all integrals in \eqref{condrec} are divergent and the random walk on a spider is null recurrent (since there is no jump at the point 1).
\item If $a<c,$ there will always be a jump at the point 1 (see \eqref{WWd}). Therefore the random walk on a spider is positive recurrent.
\end{itemize}

Finally, it is possible to see, by looking at the three-term recurrence relations \eqref{recrel1} and \eqref{recrel2}, that the entries of the matrix-valued polynomials $\bm Q_n(x)$ in \eqref{polmat} are given by
\begin{equation}\label{chebpoly}
\begin{split}
Q_{n,N}(x)&=\frac{1}{\alpha_N}\left(\frac{c}{a}\right)^{n/2}\left[2(\alpha_N-a)T_n\left(\frac{x-b}{2\sqrt{ac}}\right)+(2a-\alpha_N)U_n\left(\frac{x-b}{2\sqrt{ac}}\right)+\sqrt{\frac{a}{c}}(b-\alpha_0)U_{n-1}\left(\frac{x-b}{2\sqrt{ac}}\right)\right],\\
Q_{n,N}^{(0)}(x)&=-\frac{1}{\alpha_N}\left(\frac{c}{a}\right)^{(n-1)/2}U_{n-1}\left(\frac{x-b}{2\sqrt{ac}}\right),\\
Q_{n,k}(x)&=\left(\frac{c}{a}\right)^{n/2}U_n\left(\frac{x-b}{2\sqrt{ac}}\right),\quad Q_{n,k}^{(0)}(x)=-\left(\frac{c}{a}\right)^{(n+1)/2}U_{n-1}\left(\frac{x-b}{2\sqrt{ac}}\right),\quad k=1,\ldots,N-1,\\
\end{split}
\end{equation}
where $(T_n)_n$ and $(U_n)_n$ are the Chebychev polynomials of the first and second kind, respectively.

%\begin{remark}
%The polynomials $(Q_{n,N})_{n\ge0}$ in \eqref{chebpoly} can be written in terms of \emph{perturbed Chebychev polynomials} (see pages 204-205 of \cite{Ch}). Indeed, these perturbed Chebychev polynomials $(P_n)_{n\ge0}$ are defined in terms of the three-term recurrence relation
%$$
%P_0(x)=1,\quad P_1(x)=\delta x-\gamma,\quad xP_n(x)=\frac{1}{2}P_{n+1}(x)+\frac{1}{2}P_{n-1}(x),\quad n\geq1,\quad \delta\neq0.
%$$
%These polynomials can also be written in terms of Chebychev polynomials of the first and second kind (see formula (13.4) of \cite{Ch}). If we use the well-known relation $U_{n-2}(x)=U_n(x)-2T_n(x)$ in (13.4) of \cite{Ch} we obtain
%$$
%P_n(x)=(2-\delta)T_n(x)-(1-\delta)U_n(x)-\gamma U_{n-1}(x),\quad n\geq0.
%$$
%A direct identification with the expression of the polynomials $(Q_{n,N})_{n\ge0}$ in \eqref{chebpoly} shows that we need to choose
%$$
%\delta=\frac{2a}{\alpha_N},\quad \gamma=\sqrt{\frac{a}{c}}\frac{\alpha_0-b}{\alpha_N},
%$$
%in order to relate $(Q_{n,N})_{n\ge0}$ with the perturbed Chebychev polynomials $(P_n)_{n\ge0}$. Therefore we obtain
%$$
%Q_{n,N}(x)=\left(\frac{c}{a}\right)^{n/2}P_n\left(\frac{x-b}{2\sqrt{ac}}\right),\quad n\ge0.
%$$
%\end{remark}

\smallskip

Let us now apply Theorem \ref{thm1} and see under what conditions we can perform a reflecting-absorbing factorization of the form $\bm P=\bm P_R\bm P_A$ for this example as the one we saw in Section \ref{sec3}. The continued fractions in \eqref{Hcf} can be written as $H_m=\alpha_m/H$, where
$$
H=1-\cFrac{c}{1}-\cFrac{a}{1}-\cFrac{c}{1}-\cFrac{a}{1}-\cdots.
$$
$H$ is a continued fraction of period 2, and the explicit value is given by
$$
H=\frac{1}{2}\left(1+a-c+\sqrt{(1+c-a)^2-4c}\right),
$$
as long as $a\leq(1-\sqrt{c})^2$. Therefore, after rationalizing, we get
$$
H_m=\frac{\alpha_m}{2a}\left(1+a-c-\sqrt{(1+c-a)^2-4c}\right),\quad m=1,\ldots,N.
$$
According to Theorem \ref{thm1}, the stochastic reflecting-absorbing factorization will be possible if and only if 
\begin{equation}\label{betam1}
\beta_m\geq\frac{\alpha_m}{2a}\left(1+a-c-\sqrt{(1+c-a)^2-4c}\right),\quad m=1,\ldots,N,
\end{equation}
and from $\sum_{m=1}^N H_m<1$ we need to have
\begin{equation}\label{alpha01}
\alpha_0>\frac{1}{2}\left(1-a+c-\sqrt{(1+c-a)^2-4c}\right).
\end{equation}
For instance, for $N=3$, the set of values
$$
a=\frac{1}{5},\quad b=\frac{11}{20},\quad c=\frac{1}{4},\quad \alpha_0=\frac{1}{2},\quad\alpha_1=\frac{1}{8},\quad\alpha_2=\frac{1}{6},\quad\alpha_3=\frac{5}{24},
$$
gives a stochastic reflecting-absorbing factorization if and only if
$$
\beta_1\geq\frac{19-\sqrt{41}}{64}\sim0.196826\cdots,\quad \beta_2\geq\frac{19-\sqrt{41}}{48}\sim0.262434\cdots,\quad\beta_3\geq\frac{95-5\sqrt{41}}{192}\sim0.328043\cdots.
$$

Under conditions \eqref{betam1} and \eqref{alpha01} we can perform a stochastic discrete Darboux transformation given by $\widetilde{\bm P}=\bm P_A\bm P_R$. This new block tridiagonal matrix $\widetilde{\bm P}$ gives rise to a family (depending on $N$ free parameters $\beta_1,\ldots,\beta_N$) of Markov chains $\{\widetilde S_n,n=0,1,\ldots\}$ on a spider $\mathbb{S}_N$ with coefficients given by \eqref{ABCg}. As we mentioned in Section \ref{ssDarboux}, $\widetilde{\bm P}$ is ``almost'' a birth-death chain on a spider, since we have to add extra probability transitions between the first states of each leg (see Figure \ref{fig3}). A direct computation shows that the new coefficients $\widetilde{\bm B}_0$ and $\widetilde{\bm A}_0$ of $\widetilde{\bm P}$ (see \eqref{ABCg}), which give the extra transitions between the first states of each leg, are given by
$$
\widetilde{\bm B}_0=\left(
\begin{array}{ccccc}
\beta_0&\beta_1&\beta_2&\cdots&\beta_{N-1}\\
\beta_0\left(1-\D\frac{\alpha_1}{\beta_1}\right)&\beta_1-\alpha_1+\D\frac{c\alpha_1}{\beta_1-\alpha_1}&\beta_2\left(1-\D\frac{\alpha_1}{\beta_1}\right)&\cdots&\beta_{N-1}\left(1-\D\frac{\alpha_1}{\beta_1}\right)\\
\beta_0\left(1-\D\frac{\alpha_2}{\beta_2}\right)&\beta_1\left(1-\D\frac{\alpha_2}{\beta_2}\right) &\beta_2-\alpha_2+\D\frac{c\alpha_2}{\beta_2-\alpha_2}&\cdots&\beta_{N-1}\left(1-\D\frac{\alpha_2}{\beta_2}\right)\\
\vdots&\vdots&&\ddots&\vdots\\
\beta_0\left(1-\D\frac{\alpha_{N-1}}{\beta_{N-1}}\right)&\beta_1\left(1-\D\frac{\alpha_{N-1}}{\beta_{N-1}}\right)&\cdots&&\beta_{N-1}-\alpha_{N-1}+\D\frac{c\alpha_{N-1}}{\beta_{N-1}-\alpha_{N-1}}\\
\end{array}
\right),
$$
$$
\widetilde{\bm A}_0=\left(
\begin{array}{ccccc}
\beta_N&0&0&\cdots&0\\
\beta_N\left(1-\D\frac{\alpha_1}{\beta_1}\right)&\D\frac{\alpha_{1}}{\beta_{1}}-\D\frac{c\alpha_{1}}{\beta_{1}-\alpha_{1}}&0&\cdots&0\\
\beta_N\left(1-\D\frac{\alpha_2}{\beta_2}\right)&0 &\D\frac{\alpha_{2}}{\beta_{2}}-\D\frac{c\alpha_{2}}{\beta_{2}-\alpha_{2}}&\cdots&0\\
\vdots&\vdots&&\ddots&\vdots\\
\beta_N\left(1-\D\frac{\alpha_{N-1}}{\beta_{N-1}}\right)&0&\cdots&&\D\frac{\alpha_{N-1}}{\beta_{N-1}}-\D\frac{c\alpha_{N-1}}{\beta_{N-1}-\alpha_{N-1}}\\
\end{array}
\right).
$$
Finally, the weight matrix $\widetilde{\bm W}$ associated with $\widetilde{\bm P}$ is given by \eqref{spdtr}, i.e.
\begin{equation*}\label{spdtrex}
\widetilde{\bm W}(x)=\bm S_0\left(\frac{\bm W(x)}{x}+\left[\bm X-\bm M_{-1}\right]\delta_0(x)\right)\bm S_0^T,
\end{equation*}
where
$$
\bm S_0=\left(
\begin{array}{ccccc}
1&0&0&\cdots&0\\
1-\D\frac{\alpha_1}{\beta_1}&\D\frac{\alpha_{1}}{\beta_{1}}&0&\cdots&0\\
1-\D\frac{\alpha_2}{\beta_2}&0 &\D\frac{\alpha_{2}}{\beta_{2}}&\cdots&0\\
\vdots&\vdots&&\ddots&\vdots\\
1-\D\frac{\alpha_{N-1}}{\beta_{N-1}}&0&\cdots&&\D\frac{\alpha_{N-1}}{\beta_{N-1}}\\
\end{array}
\right),
$$
the weight matrix $\bm W=\bm W_c+\bm W_d$ is given by  \eqref{WWc} and \eqref{WWd}, $\bm X$ is the symmetric matrix given by \eqref{XXij} and from \eqref{sttrW} we have
\begin{equation*}\label{Mm1}
\bm M_{-1}=\left(
\begin{array}{c|c}
\mu_{11}&\mu_{12}\vec{\bm e}\,^T_{N-1}\\
\hline
\mu_{12}\vec{\bm e}_{N-1}&\bm\mu_{22}
\end{array}
\right),
\end{equation*}
where
\begin{align*}
\mu_{11}&=\frac{\alpha_0(1+a-c)-b-(1-\alpha_0)\sqrt{\sigma_+\sigma_-}}{2\left[c-\alpha_0(1-a+c-\alpha_0)\right]},\quad \mu_{12}=\frac{2c-\alpha_0(1-a+c)+\alpha_0\sqrt{\sigma_+\sigma_-}}{2\left[c-\alpha_0(1-a+c-\alpha_0)\right]},\\
\bm\mu_{22}&=\frac{b-\sqrt{\sigma_+\sigma_-}}{2a}\bm\alpha_D^{-1}+\frac{\alpha_0(b-a(1-a+c))-bc-(\alpha_0(1-a)-c)\sqrt{\sigma_+\sigma_-}}{2a\left[c-\alpha_0(1-a+c-\alpha_0)\right]}\vec{\bm e}_{N-1}\vec{\bm e}_{N-1}^T,
\end{align*}
The matrix-valued polynomials $(\widetilde{\bm Q}_n)_{n\geq0}$ orthogonal with respect to $\widetilde{\bm W}$ can be computed from \eqref{polmatdar} and can also be written as combinations of Chebychev polynomials of the first and second kind (see \eqref{chebpoly}).

\end{document}